\documentclass[a4paper]{amsart}
\usepackage{amsfonts}
\usepackage[colorlinks=true,citecolor=blue]{hyperref} %
\usepackage{mathptmx}
\usepackage{eucal}
\usepackage{graphicx}
\usepackage{mathrsfs}
\usepackage{enumerate}
\usepackage{verbatim}

\usepackage{xcolor}
\newtheorem{thm}{Theorem}[section]
\newtheorem{cor}[thm]{Corollary}

\theoremstyle{definition}

\newtheorem{rem}[thm]{Remark}
\numberwithin{equation}{section}

\newcommand{\N}{\mathbb{N}}

\newcommand{\eps}{\varepsilon}

\begin{document}
\title{Corrigendum to ``Syndetically proximal pairs'' [J. Math. Anal. Appl. 379 (2011)
656--663]}

\author[J. Li]{Jian Li}
\address[J. Li]{Department of Mathematics,
    Shantou University, Shantou, 515063, Guangdong, China
    -- and --
    Guangdong Provincial Key Laboratory of Digital Signal and Image Processing
    Techniques, Shantou University, Shantou, Guangdong 515063, China}
\email{lijian09@mail.ustc.edu.cn}

\author[T.K.S. Moothathu]{T.K. Subrahmonian Moothathu}
\address[T.K.S. Moothathu]{School of Mathematics and Statistics, University of Hyderabad,
Hyderabad 500 046, India.}
\email{tksubru@gmail.com}
\author[P. Oprocha]{Piotr Oprocha}
\address[P. Oprocha]{AGH University of Science and Technology, Faculty of Applied
    Mathematics, al.
    Mickiewicza 30, 30-059 Krak\'ow, Poland
    -- and --
    National Supercomputing Centre IT4Innovations, Division of the University of Ostrava,
    Institute for Research and Applications of Fuzzy Modeling,
    30. dubna 22, 70103 Ostrava,
    Czech Republic}
\email{oprocha@agh.edu.pl}

\begin{abstract}
We give a counterexample to Theorem 9 in
``Syndetically proximal pairs''[J. Math. Anal. Appl. 379 (2011) 656--663].
We also provide sufficient conditions for the conclusion of Theorem 9 to hold.
\end{abstract}

\maketitle
\section{Introduction}

The reader not familiar with the theory of entropy, in particular this theory in the context
of one-dimensional dynamics, is refereed to monographs \cite{BC,SR,ALM}.

One of the celebrated results in the one-dimensional dynamics is that if a continuous
interval map $f\colon [0,1] \to[0,1]$ has positive topological entropy 
then there is a horseshoe in its iteration \cite{Mi1,Mi2}.
An important consequence of this result is that it has a factor map to the full shift.
More specifically, we have the following useful result, see \cite[Theorem 8]{M11} for this version.
Note that the one-sided full shift dynamical system on the alphabet
$\{0,1,2,\dotsc,m-1\}$ is denoted by $(\Sigma_m,\sigma)$.
\begin{thm}\label{thm:semi-conjugate}
Let $f\colon [0,1]\to[0,1]$ be a continuous map with positive topological entropy.
Then there exist $n\in\mathbb{N}$, an $f^n$-invariant closed set $X\subset [0,1]$,
and a continuous surjection $\phi\colon X\to\Sigma_2$ such that
\begin{enumerate}
\item $\phi\circ f^n(x)=\sigma\circ \phi(x)$ for every $x\in X$.
\item $|\phi^{-1}(y)|\leq 2$ for every $ y\in\Sigma_2$.
\item The set $\{y\in\Sigma_2\colon |\phi^{-1}(y)|>1\}$ is at most countable.
\end{enumerate}
\end{thm}

It is claimed \cite{M11} that the map $\phi$ in Theorem \ref{thm:semi-conjugate}
can be chosen to be a homeomorphism. The following result is Theorem 9 in \cite{M11}.

\begin{thm}\label{thm:conjugate}
Let $f\colon[0,1]\to[0,1]$ be a continuous map with positive topological entropy and let $m\geq 2$.
Then there exist $n\in\mathbb{N}$, an $f^{2^n}$-invariant closed set $X\subset [0,1]$, and
a homeomorphism $\phi\colon X\to\Sigma_m$ such that $\phi\circ f^{2^n}(x)=\sigma\circ\phi(x)$
for every $x\in X$.
\end{thm}

After checking the proof of Theorem~\ref{thm:conjugate} in \cite{M11} carefully,
we found some gaps in the proof and later realized that we are able to construct a
counterexample to the statement of Theorem~\ref{thm:conjugate}. 
Strictly speaking, we have the following result.

\begin{thm}\label{thm:example}
There exists a surjective continuous map $f\colon [0,1]\to [0,1]$  with positive topological entropy
such that for every $n\in \N$ and every $f^n$-invariant closed set $X\subset [0,1]$
the map $f^n|_X$ is not topologically conjugate to the full shift dynamical system $(\Sigma_2,\sigma)$.
\end{thm}

\begin{rem}
Even though Theorem~\ref{thm:conjugate} turns out to be false,
several results using it remain valid, for examples Theorems 10 and 11 in \cite{M11},
Theorem 6.1 in \cite{LO13}, Theorem 6.7 in \cite{LT14}.
It seems sufficient to use Theorem~\ref{thm:semi-conjugate} instead of Theorem~\ref{thm:conjugate}
and some standard techniques such as in the proof of Theorem 5.17 in \cite{SR}.
\end{rem}

We also give sufficient conditions for the conclusion of Theorem~\ref{thm:conjugate} to hold.
These conditions cover quite a large class of interval maps.

\begin{thm}\label{thm:conditions}
Let $f\colon [0,1]\to [0,1]$ be a continuous map.
If $f$ is transitive, then there exist $n\in \N$ and an $f^n$-invariant closed set $X\subset [0,1]$
such that $f^n|_X$ is topologically conjugate to the shift dynamical system $(\Sigma_2,\sigma)$.
\end{thm}

Recall that a point $x\in [0,1]$ is \emph{equicontinuous} if for every $\varepsilon>0$
there exists an open neighborhood $U$ of $x$ such that $diam(f^n(U))<\varepsilon$ for all $n\geq 0$.

\begin{cor}\label{cor:conditions1}
Let $f\colon [0,1]\to [0,1]$ be a continuous map.
If the set of equicontinuity points of $f$ fails to be dense in $[0,1]$,
then there exist $n\in \N$ and an $f^n$-invariant closed set $X\subset [0,1]$
such that $f^n|_X$ is topologically conjugate to the shift dynamical system $(\Sigma_2,\sigma)$.
\end{cor}
\begin{proof}
For $k\in \N$, let $S_k$ be the collection of all $x\in [0,1]$ with the following property:
for every open neighborhood $U$ of $x$,
there is $n\in \N$ such that $diam(f^n(U))\ge 1/k$. Then it may be seen that each $S_k$
is closed (and also $f$-invariant; but this we do not need).
Note that $\bigcup_{k=1}^\infty S_k$ is the complement of the set of equicontinuity points
of $f$ (in other words, $\bigcup_{k=1}^\infty S_k$ is the set of sensitivity points of $f$).
By the hypothesis of the Corollary, $\bigcup_{k=1}^\infty S_k$
contains a nondegenerate interval. Since $S_k$'s are closed, 
we conclude by Baire category theorem that $int(S_k)\ne \emptyset$ for some $k\in \N$. 
Now by Proposition 2.40 of \cite{SR}, there exists a cycle $L_1,\ldots,L_p$ of
closed intervals such that $f$ restricted to $L_1\cup \cdots \cup L_p$ is transitive.
Then $f^p$ restricted to $L_1$ must be transitive.
It is enough to apply Theorem~\ref{thm:conditions} to the restriction of $f^p$ to $L_1$.
\end{proof}

\begin{cor}\label{cor:1}
Let $f\colon [0,1]\to [0,1]$ be a continuous map with positive entropy.
If $f$ has a dense set of periodic points,
then there exist $n\in \N$ and an $f^n$-invariant closed set $X\subset [0,1]$ such
that $f^n|_X$ is topologically conjugate to the shift dynamical system $(\Sigma_2,\sigma)$.
\end{cor}

\begin{proof}
Since $h(f)>0$, $f^2$ cannot be identity. Hence by Proposition 3.8 of \cite{SR},
either $f$ or $f^2$ must be transitive on a nondegenerate closed interval $J\subset [0,1]$.
Apply Theorem~\ref{thm:conditions} to $f|_J$ or $f^2|J$.
\end{proof}

\section{Proofs of the main results}

\begin{proof}[Proof of Theorem~\ref{thm:example}]
In fact we will construct a map $f\colon [0,2]\to [0,2]$ which after normalization to a map
$\tilde f\colon [0,1]\to [0,1]$ is an example as required.

Start by considering a map $g\colon [0,1]\to [0,1]$ defined by
$g(x)=\min \{1, 3/2-|3x-3/2|\}.$
In other words, $g$ is a tent map with slope $\pm3$ and flattened top.
Note that there is a Cantor set $C\subset [0,1]$ such that $(C,g)$ is conjugated
(by a homeomorphism $\eta\colon C\to \Sigma_2$) with $(\Sigma_2,\sigma)$ and
if $x\in [0,1]\setminus C$ then there is $n\geq 0$ such that $g^n(x)=0$.

Let $\{C_i\}_{i\in A}$ be the family of all closed subsets of $C$ such that $g^{k_i}(C_i)=C_i$
for some ${k_i}>0$ and $(\eta(C_i),\sigma^{k_i})$ is a non-trivial mixing sofic shift
in the higher power block representation of the full shift $(\Sigma_{2^{k_i}},\sigma)$
which is conjugated  with $(\Sigma_2,\sigma^{k_i})$.
Since every sofic shift has a labeled graph presentation (see \cite[\S 3.1]{ISDC}),
the set $A$ is countable.
Note that it may happen that for some $i\neq j$ we have $C_i \subset C_j$ or
even $C_i=C_j$ but $k_i\neq k_j$.

Note that each $(C_i,g^{k_i})$ is non-trivial mixing.
There is a countable sequence $\{x_i\}_{i\in A}$ of points (not necessarily pairwise distinct)
such that $x_i\in C_i$ is a transitive point of  $(C_i,g^{nk_i})$ for all $n\geq 1$.
Since $C_i$ is perfect, for every nonempty open set $U\subset C_i$ there are points $a,b,c\in U$, $a<b<c$
and $c-a<\eps$. Then $V=(a,c)\cap C_i$ is a nonempty open subset of $C_i$ and for every $y\in V$
we have $(y-\eps,y)\cap C_i\neq \emptyset$ and $(y,y+\eps)\cap C_i\neq \emptyset$.
This immediately shows that
the set of points $y\in C_i$ such that for every $\eps>0$ we have
$(y-\eps,y)\cap C_i\neq \emptyset$ and $(y,y+\eps)\cap C_i\neq \emptyset$ is residual.
Therefore we can require that $x_i$ is such that for any $\eps>0$ and $n>0$ there are $s,t>0$
such that $g^{ns}(x_i)\in (x_i-\eps,x_i)$ and $g^{nt}(x_i)\in (x_i,x_i+\eps)$.

We will perform a construction similar to the standard Denjoy extension of irrational
rotation on the unit circle (see e.g. \cite[Proposition 4.4.4]{HK03}).
First observe that by the definition $g^j(x_i)\neq 0$ for any $j\geq 0$ and $i\in A$ and
hence the set
$$
D=\bigcup_{i\in A} \bigcup_{k\geq 0}g^{-k}(\{g^j(x_i)\colon j=0,1,\dotsc\})
$$
is countable, because if $z\neq 1$ then $g^{-1}(z)$ has exactly two elements.
Furthermore $g(D)=D$, $g^{-1}(D)=D$ and $D\subset (0,1/3)\cup (2/3,1)$.
Enumerate elements of $D$, say $D=\{z_j\}_{j\in \N}$.
Extend $[0,1]$ to $[0,2]$ by inserting in place of each $z_j$ an interval $I_j$ of length $2^{-j}$.
This way we have a monotone surjective map $\pi \colon [0,2]\to [0,1]$ which is
one to one for each $x\in [0,2]\setminus \cup_{j} I_j$ and $\pi(I_j)=z_j$ for $j\in\mathbb{N}$.

We will define a map $f\colon [0,2]\to [0,2]$ in the following way.
For $x\not\in \cup_j I_j$ we put $f(x)=\pi^{-1}g(\pi (x))$.
If $x\in I_j$ then $\pi(x)=z_j\in D$ and hence $g'(z_j)=3$ or $g'(z_j)=-3$.
There exists $s\in\mathbb{N}$ such that $g(z_j)=z_s$.
Then we define $f|_{I_j}\colon I_j \to I_s$ as a homeomorphic map of constant slope
which is increasing when  $g'(z_j)=3$ and decreasing in the other case.
Observe that the map $f$ defined that way is continuous and $\pi \circ f=g\circ \pi$.

Since $f$ is an extension of $g$, the topological entropy of $f$ is also positive.
Suppose that there exist $m\in\mathbb{N}$ and an $f^m$-invariant closed set $X\subset [0,2]$
such that $(X,f^m)$ is conjugated to $(\Sigma_2,\sigma)$.
Then $(\pi(X),g^m)$ is mixing and infinite. This implies that $\pi(X)\subset C$ because if
$\pi(X)\setminus C\neq \emptyset$
then there exists an open set $U\subset \pi(X)$ and $k>0$ such that $g^{km}(U)=\{0\}$ which is
impossible.
Hence $(\eta(\pi(X)),\sigma^m)$ is a factor of $(X,f^m)$, and therefore is a sofic shift as a
factor of a shift of finite type.
Therefore there exists an $i\in A$ such that $\pi(X)=C_i$ and $m=k_i$.
There exists $r\in\mathbb{N}$ such that $z_r=x_i\in C_i$.
Let $I_r=[a,b]$. Observe that $f^j(a,b)\cap (a,b)=\emptyset$ for every $j>0$.
In particular $a$ and $b$ are asymptotic, that is $\lim_{j\to\infty}|f^j(a)-f^j(b)|=0$.
As $(X,f^m)$ is mixing and $\pi^{-1}(z_r)=I_r$, $I_r\cap X\subset \{a,b\}$.
Without loss of generality, assume that $a\in X$.
Note that the orbit of $x_i$ under $g^{k_i}$ intersects both intervals
$(x_i-\eps,x_i)$ and $(x_i,x_i+\eps)$ for every $\eps>0$.
Let $\{s_j\}$ be an increasing sequence of positive integers
such that $g^{k_is_j}(x_i)>x_i$ and $\lim_{j\to \infty} g^{k_i s_j}(x_i)=x_i$.
Without loss of generality we may also assume that $\lim_{j\to \infty} f^{k_i s_j}(a)$ exists.
Since $\pi$ is monotone and $\pi\circ f=g\circ \pi$, we get $f^{k_i s_j}(a)>b$ and so
$\lim_{j\to \infty} f^{k_i s_j}(a)=b$. This shows that $a,b\in X$.
But for any asymptotic pair $p,q$ in the one sided full shift,
there exists $n$ such that $\sigma^n(p)=\sigma^n(q)$. This implies $f^{mn}(a)=f^{mn}(b)$
which contradicts the fact that all intervals $I_j$ are nondegenerate.
\end{proof}

\begin{proof}[Proof of Theorem~\ref{thm:conditions}]
First we recall a well known fact that if $f$ is transitive but not mixing,
then there is $0<c<1$
such that $[0,c]$ is $f^2$-invariant and the restriction of $f^2$ to $[0,c]$ is mixing
(see e.g. \cite[Proposition 2.16]{SR}).
Replacing $f$ by $f^2$ if necessary, we assume that $f$ is mixing.
As $f$ has positive topological entropy (see e.g. \cite[Proposition 4.70]{SR}),
there exist $r\in \N$ and disjoint
closed intervals $J_0,J_1$ such that  $g:=f^r$ satisfies $J_0\cup J_1\subset g(J_0)\cap
g(J_1)$. Without loss of generality assume $J_0$ is to the left of $J_1$, and let $(a,b)$ be
the open interval between $J_0$ and $J_1$ (i.e., $a=\max J_0$ and $b=\min J_1$). For later
use, keep in mind that $0<a<b<1$. By the horseshoe property of $J_0$ and $J_1$, for $k\in \N$
and $w=w_1\cdots w_k\in \{0,1\}^k$, we can find closed intervals $J_w$ such that
$J_{w_1\cdots w_k}\subset J_{w_1\cdots w_{k-1}}$ and $g(J_{w_1\cdots w_k})=J_{w_2\cdots w_k}$.
For $\alpha=w_1w_2\cdots \in \Sigma_2$, let $J_\alpha=\bigcap_{k=1}^\infty J_{w_1\cdots w_k}$.
Then $J_\alpha$ is either a singleton or a nondegenerate closed interval.
If possible, let $J_\alpha$ be a nondegenerate closed interval. Since $f$ is mixing and
$0<a<b<1$, there is $n_0\in \N$ such that $(a,b)\subset f^n(J_\alpha)$ for every $n\ge n_0$.
On the other hand, $f^{rn}(J_\alpha)=g^n(J_\alpha)\subset J_0\cup J_1$ for every $n\in \N$.
This is a contradiction, and therefore $J_\alpha$ must be a singleton, say
$J_\alpha=\{x_\alpha\}$. Letting $X=\{x_\alpha:\alpha\in \Sigma_2\}$, we may check that $X$
is a $g$-invariant closed set and $(X,g)$ is topologically conjugate to $(\Sigma_2,\sigma)$
via the conjugacy $x_\alpha\mapsto \alpha$.
\end{proof}

\section*{Acknowledgements}
The authors are grateful to Sylvie Ruette for numerous discussions
on the construction of conjugacy with full shift and proofs in \cite{SR}. 
The authors would like to thank the anonymous referee 
for the careful reading and helpful suggestions.

Research of J. Li was supported in part by NSF of China (grant numbers 11401362 and 11471125).
Research of P. Oprocha was supported by National Science Centre, Poland (NCN), grant no.
2015/17/B/ST1/01259.

\end{document}